\documentclass[a4paper,10pt]{article}
\usepackage{amsmath,amsthm,amsfonts}
\usepackage{graphicx,color,epsfig}

\newtheorem{thm}{Theorem}
\newtheorem{lem}[thm]{Lemma}
\newtheorem{cor}[thm]{Corollary}
\newtheorem{prop}[thm]{Proposition}
\theoremstyle{remark}
\newtheorem{rem}[thm]{Remark}

\numberwithin{equation}{section}
\numberwithin{thm}{section}
\newcommand{\R}{\mathbb{R}}

\newcommand{\M}{\mathbb{M}}
\newcommand{\N}{\mathbb{N}}
\newcommand{\A}{\mathcal{A}}

\newcommand{\ex}{\mathrm{e}}
\newcommand{\hess}{\mathrm{Hess}}
\newcommand{\jac}{\mathrm{Jac}}
\newcommand{\Ric}{\mathrm{Ric}}
\newcommand{\dM}{\partial_\infty M}
\newcommand{\Vol}{\textup{Vol}}
\newcommand{\vol}{\textup{vol}}
\newcommand{\tr}{\textup{tr}}
\newcommand{\id}{\textup{Id}}
\newcommand{\ddt}{\frac{\partial}{\partial t}}
\begin{document}

  \begin{center}
    \Large
    \textbf{On asymptotically harmonic manifolds \\ of negative curvature}\\
    \normalsize
    \vspace{10mm}
    Philippe CASTILLON and Andrea SAMBUSETTI
  \end{center}

  \small
  \noindent\textbf{Abstract:} We study asymptotically harmonic manifolds of
  negative curvature, without any cocompactness or homogeneity assumption.
  We show that asymptotic harmonicity provides a lot of  information  on the
  asymptotic geometry of these spaces: in particular, we determine the volume
  entropy, the spectrum and the relative densities of visual and harmonic
  measures on the ideal boundary. Then, we prove an asymptotic analogue of the
  classical mean value property of harmonic manifolds, and we characterize
  asymptotically harmonic manifolds, among Cartan-Hadamard spaces of strictly
  negative curvature, by the existence of an asymptotic equivalent
  $\tau(u)\ex^{Er}$ for the volume-density of geodesic spheres (with $\tau$
  constant in case $DR_M$ is bounded). Finally, we show the existence of a
  Margulis function, and explicitly compute it, for all asymptotically
  harmonic manifolds.

  \vspace{5mm}

  \vspace{5mm}

  \normalsize
  \vspace{5mm}
%
%
%
%
\section*{Introduction}

{\em Harmonic manifolds} are  those Riemannian manifolds whose geodesic spheres
have constant mean curvature; equivalently, such that the volume density
function, in normal coordinates at any point $x$,  only depends on the distance
$d(x, \cdot)$. Another equivalent condition is that the mean-value property
$$
F(x_0) = \frac{1}{\vol (S_{x_0}(R))} \int_{S_{x_0}(R)} F(x) dv_{S_{x_0}(R)}
$$
holds for all harmonic functions $F$ on $M$ (cf. \cite{Besse}).

In 1944, A. Lichnerowicz conjectured (and proved in dimension 4) that the rank
one symmetric spaces (denoted ROSS, in the sequel) are the only harmonic
manifolds. If this was proved to be true for compact simply connected manifolds
(cf. \cite{Szabo})  and for negatively curved Cartan-Hadamard manifolds admitting
compact quotients (cf. \cite{Be-Co-Ga} section 9.C), E. Damek and F. Ricci
constructed harmonic homogeneous manifolds which are not ROSS
(cf. \cite{Damek-Ricci}).
Since then, J.~Heber proved that Damek-Ricci spaces and ROSS are the only
homogeneous harmonic manifolds (cf. \cite{Heber}), and then further relations
between harmonicity, volume growth and Gromov hyperbolicity have been studied
(cf. \cite{Ranjan-Shah,Knieper2}).

In several of these works, an asymptotic version of harmonicity naturally appears
(cf. \cite{Foulon-Labourie,Heber}) : a Cartan-Hadamard manifold $M$ is {\em
asymptotically harmonic} if its horospheres have constant mean curvature $h$.
This notion was introduced by F.~Ledrappier in \cite{Ledrappier}, and was mainly
studied in the {\em cocompact} case (i.e. when the space admits  compact
quotients). F. Ledrappier  proved that, within these spaces, asymptotic
harmonicity is equivalent to the condition $\inf \sigma(\Delta) = \frac{E^2}{4}$
(where $\sigma (\Delta)$ denotes the spectrum of the Laplacian of $M$, and $E$
its volume-entropy); moreover, he showed that if  $M$ is asymptotically harmonic,
then $E=nh$ and $\inf \sigma(\Delta) = \frac{n^2h^2}{4}$. 
It was then proved  (as a consequence of the work of Y. Benoist, P. Foulon and
F.~Labourie on the geodesic flow of asymptotically harmonic spaces
\cite{Be-Fo-La,Foulon-Labourie} and the characterization of locally symmetric
spaces by their volume entropy due to G. Besson, G. Courtois and S. Gallot 
\cite{Be-Co-Ga})  that  the ROSS are the only asymptotically harmonic manifolds
among cocompact, negatively curved Cartan-Hadamard spaces.
On the other hand, in \cite{Connell} necessary and sufficient conditions are
given in order that a {\em homogeneous}, negatively curved Cartan-Hadamard
manifold is asymptotically harmonic; however, as far as the authors know, the
problem whether any asymptotically harmonic manifold is a ROSS or a Damek-Ricci
space is still open in this class.
Recently, it was also proved that, in dimension 3, the only asymptotically
harmonic Cartan-Hadamard manifold of strictly negative curvature  is the
hyperbolic space (cf. \cite{He-Kn-Sh,Schroeder-Shah}).

The aim of this paper is to show that, for Cartan-Hadamard manifolds of strictly
negative curvature of any dimension, even  without any cocompactness or
homogeneity assumption, asymptotic harmonicity provides a lot of information  on
the asymptotic geometry. 
In view of \cite{Ledrappier}, we are naturally  interested  in the volume
entropy, the spectrum and the relations between visual and harmonic measures on
the ideal boundary of a general asymptotically harmonic manifold.
In particular,  in section \S3, we show rigidity of Cartan-Hadamard
asymptotically harmonic manifolds under suitable curvature bounds ({\em Corollary
\ref{cor_rigidity}}),  we determine the volume entropy and the spectrum (cf. {\em
Theorems \ref{thm_entropy} \& \ref{thm_spectrum}})  and, when the curvature is
negatively pinched, we find sharp upper and lower bounds for the volume-growth of
the horospheres ({\em Theorem \ref{thm_growth}} and ff. {\em Remarks
\ref{rem_optimal} \& \ref{rem_lower}}). Moreover, we prove  an asymptotic
analogue of the classical mean-value property holding on harmonic manifolds ({\em
Theorem \ref{thm_meanvalue}}). 
In section \S4, we characterize asymptotically harmonic manifolds  as  those
manifolds whose volume form, in normal coordinates, is asymptotically equivalent
to a function $\tau(u)\ex^{ER}$, for some positive function $\tau$ on $SM$ ({\em
Theorem \ref{thm_characterization}}); then, we show  that the function $\tau$ is
constant if $DR_M$ (the derivative of the Riemann tensor) is bounded ({\em
Proposition  \ref{prop_tau-properties}(ii)}). \linebreak 
In \S5  we prove the existence of a Margulis function
({\em Proposition \ref{propmargulis}}), we explicitly compute it for all
asymptotically harmonic manifolds, and we find the relative densities of visual
and harmonic measures on the ideal boundary  
({\em Proposition \ref{prop_densities}});  we also show  that they coincide when
$DR_M$ is bounded. This result is to compare to what is known in the cocompact
and homogeneous cases, where  coincidence of two of the three natural families of
measures  on the ideal boundary  (visual, harmonic and Patterson-Sullivan
measures) forces, respectively in the two cases, symmetry  and asymptotic
harmonicity of the manifold (cf. \cite{Ledrappier, Ledrappier2, Yue, Yue2,
Connell}); unfortunately, a similar characterization  for general asymptotically
harmonic Cartan-Hadamard manifolds is still missing.
 
The main tools we use are a comparison lemma for the second fundamental forms of
two tangent spheres, which is proved in
section 2, and the Riccati equation. The first section is devoted to notations
and  preliminary results.

\vspace{2mm}

{\em We thank professor S. Gallot for his suggestions and encouragement, and
professor G. Knieper for explaining us the expression of the function $\tau$ in
terms of Jacobi tensors.}

%
%
%
%
\section{Notations}

Unless otherwise stated, throughout all the paper $(M, g)$ will always be a
{\em Cartan-Hadamard manifold} (CH-manifold, for short) of dimension $n+1$, i.e.
a complete, simply connected Riemanniann manifold with nonpositive curvature.

The {\em ideal boundary} of $M$, denoted $\dM$, is the set of equivalent classes
of geodesic rays, $\gamma$ and $\sigma$ being equivalent if
$\sup\{d(\gamma(t),\sigma(t))\ |\ t\ge 0\}<\infty$ (cf. \cite{Bridson-Haefliger}
definition II.8.1). For $\xi\in\dM$, $\lim_{t\to+\infty}\gamma(t)=\xi$ will mean
that $\xi$ is the equivalence class defined by $\gamma$. The cone topology
turn $M\cup\dM$ into a compact manifold with boundary
(cf \cite{Bridson-Haefliger} definition II.8.6).

For $\xi\in\dM$ and $x\in M$, the {\em Busemann function} $b_{\xi,x}$, centered
at $\xi$ and vanishing at $x$, is defined by
$b_{\xi,x}(y)=\lim_{t\to+\infty}(d(y,\gamma(t))-t)$, where $\gamma$ is the unique
geodesic such that $\gamma(0)=x$ and $\lim_{t\to+\infty}\gamma(t)=\xi$. Two
Busemann functions centered in the same point at infinity differ from a constant ;
in many situations, we only need to know the Busemann functions up to a constant,
and we shall note $b_\xi$ some Busemann function centered in $\xi$.
Busemann functions are Lipschitz and, on CH-manifolds, they are at least $C^2$,
cf. \cite{Heintze-Hof}. 

The {\em horospheres} centered in $\xi\in \dM$ are the level hypersurfaces of
$b_\xi$: $H_{\xi} (t) = \{ x \in M \, | \, b_{\xi} (x) = t\}$; we shall also use
the convenient notation $H_{\xi} (x)$ for the horosphere centred at $\xi$ and
passing through $x$. As the Busemann functions are limit of distance functions,
the horospheres centered in $\xi$ are (locally) limit of spheres whose centers
tends to $\xi$. Since $|\nabla b_\xi|=1$  and the gradient lines of $b_\xi$ are
the geodesics $\gamma$ such that $\lim_{t\to-\infty}\gamma(t)=\xi$, we can define 
the {\em inner} unit vector field of horospheres centred at $\xi$ as 
$\nu = - \nabla b_\xi$   (i.e.  $\nu$ points towards the center $\xi$ of the
horosphere).

For a general hypersurface $N$ of $M$,   $\vec{\A}^N$ denotes its {\em second}
(vector valued)  {\em fundamental form} ; that is, for
$u,v\in T_xN$, $\vec{\A}^N (u,v)$ is the component of $D^M_uV$ normal   to $N$,
where $D^M$ is the connection of $M$ and $V$ extends $v$ in a neighborhood of
$x$. 
Associated to the choice of a unit normal vector field $\nu$ to $N$ we then have
the {\em second} (scalar)  {\em fundamental form}
$\A^N = \langle \vec{\A}^N,\nu \rangle$ and   the  {\em shape operator}
$A^N \in End(T_xN)$, defined by
$\langle A^N u,v \rangle = \langle D^M_u\nu,v \rangle= -\langle \vec{\A}^N(u,v),\nu \rangle$.
The {\em mean curvature vector} of $N$ at $x$ is
$\vec{h}^N (x)=\frac{1}{n} Tr \vec{\A}^N (x)$, while the (scalar) mean curvature,
associated with $\nu$, is $h^N  =\langle \vec{h}^N ,\nu \rangle$.

A manifold $M$ is called {\em asymptotically harmonic}  if all its horospheres
have constant mean curvature $h$. The curvature of $M$ being nonpositive, the
horospheres are convex and we have $f\ge 0$ when choosing $\nu$ pointing
to the center of the horosphere.

\subsubsection*{Hessian and Laplacian of Busemann functions}

The second fundamental form naturally appears when restricting a function
to a submanifold :
\begin{prop}\label{prop_hessian}
    Let $i:N\to M$ be an isometric immersion, let $F:M\to \R$ be a smooth
    function and let $f=F_{|_N}$ be its restriction to $N$.

    For all $x\in N$ and all $u,v\in T_xN$ we have
    $$
    (\hess^N f)(u,v) = (\hess^M F)(u,v) + \langle\nabla^MF,\vec{\A}^N(u,v)\rangle
    $$
\end{prop}
\begin{proof}
    The proof is standard.
\end{proof}

As a consequence,   the Hessian of the Busemann function is given by the second
scalar fundamental form of its horospheres, with respect to the inner  normal
vector field $\nu$; taking the trace we get
$\Delta b_\xi (y)= - Tr (\hess_y  b_{\xi}) = -nh_{\xi}  (y)$, where $h_{\xi} (y)$
is the mean curvature at $y$ of the horosphere centered in $\xi$ passing through
$y$, with respect to $\nu$. (Similarly, the second fundamental form of spheres is
the Hessian of the distance function to the center and the Laplacian of the
distance from a point $x$ gives the mean curvature of the spheres centered in $x$).

It follows, by the regularity theory of solutions of elliptic equations, that for
asymptotically harmonic manifolds Busemann functions and horospheres  are at least
as regular as the metric (whereas they are known to be real analytic on harmonic
manifolds, cf. \cite{Ranjan-Shah2}). Moreover, it is then straightforward to check
that,  for any asymptotically harmonic manifold $M$ with horospheres of mean
curvature $h$, the function $f(y)=\ex^{-nhb_{\xi}(y)}$ is  harmonic.

\subsubsection*{The Riccati equation}
 
Let $\xi\in\dM$ and  $\gamma$ be a geodesic such that
$\lim_{t\to-\infty}\gamma(t)=\xi$. For each $t$, let $A_{\xi}(t)$ be the shape
operator of the horosphere centered in $\xi$ passing through $\gamma(t)$, with
respect to the inner unit vector field   $\nu = -\nabla b_{\xi} =-\gamma'(t)$;
this family of operators satisfies the Riccati equation (cf. \cite{Karcher} \S 1.3):
\begin{equation} \label{eqn_Riccati}
	A_{\xi}'(t) + A_{\xi}^2(t) + R_M(\dot{\gamma}(t),.)\dot{\gamma}(t) = 0
\end{equation}
where $R_M$ is the Riemann tensor of $M$.

%
%
%
%
\section{Comparison of spheres on CH-manifolds}

In the sequel, we note $\M^n(-a^2)$ the simply connected Riemannian manifold
with constant sectional curvature $-a^2$, and
we shall note $C_a$ and $\cot_a$ the functions defined by:
$$
C_a(s) = \left\{\begin{array}{lr}
					\!\!\frac{1}{a^2}(\cosh(as)-1) & \mbox{if } a>0 \\
					\!\!\frac{s^2}{2} & \mbox{if } a=0
				\end{array}\right.
\mbox{ and }
\cot_a(s) = \left\{\begin{array}{lr}
						\!\!a\coth(as) & \mbox{if } a>0 \\
						\!\!\frac{1}{s} & \mbox{if } a=0
					\end{array}\right.
$$

%
\subsubsection*{Comparison of triangles}
%
When assuming a sectional curvature upper bound $K_M\le-a^2$ for $M$, the
classical Toponogov theorem (cf. \cite{Karcher}) implies that,
given two edges of a triangle in $M$
with angle $\alpha$ at the common vertex, then the third edge is larger than
the one of a triangle in the model space $\M^2(-a^2)$ with the same lengths for
the first two edges
and the same angle at the common vertex. The following lemma  is a slight
modification
of this result, where we compare the ratio of (some function of) the lengths of
the third edge
and of an ``intermediate edge''.

\begin{lem}[Triangle comparison with   curvature upper bound] 
  Let $M$ be a CH-manifold with $K_M\le-a^2\le 0$.
  Let $(xyz)$ and $(\tilde{x}\tilde{y}\tilde{z})$ be triangles in $M$ and
  $\M^2(-a^2)$ respectively, such that
  $r_1=d(x,y)=d(\tilde{x},\tilde{y})$,  $r_2=d(x,z)=d(\tilde{x},\tilde{z})$ and 
  $\alpha =\angle_x(y,z)=\angle_{\tilde{x}}(\tilde{y},\tilde{z})$.
  Moreover, for $\theta\in]0,1[$  let $p$, $q$ and $\tilde{p}$, $\tilde{q}$ be
  respectively the points on the geodesic segments $xy$, $xz$ and
  $\tilde{x}\tilde{y}$, $\tilde{x}\tilde{z}$  such that
  $d(x,p)=d(\tilde{x},\tilde{p})=\theta r_1$ and
  $d(x,q)=d(\tilde{x},\tilde{q})=\theta r_2$ (cf. figure \ref{fig-triangles}).
  Then: 
  $$
  \frac{C_a(d(y,z))}{C_a(d(p,q))} \ge
  \frac{C_a(d(\tilde{y},\tilde{z}))}{C_a(d(\tilde{p},\tilde{q}))} =
  F_{a} (r_1,r_2, \alpha, \theta).
  $$
  \label{lem_triangle_upper}
\end{lem}

\begin{rem}
  By the cosine formula in $\M^2(-a^2)$ (cf. \cite{Bridson-Haefliger}
  proposition I.2.7) we know that the right-hand side of the above inequality
  only depends on the lengths $r_1, r_2, \alpha$ and $\theta$, whence the
  existence of the function $F_{a}$.

  When $a=0$ we have $F_0=\frac{1}{\theta^2}$ and lemma \ref{lem_triangle_upper}
  is a direct consequence of the convexity of the distance function in
  $CAT(0)$-spaces (cf. \cite{Bridson-Haefliger} proposition II.2.2).
  
  When $a>0$ we find  : 
  $$
  F_a(r_1,r_2,\alpha,\theta) =
  \frac{\cosh(ar_1)\cosh(ar_2)-\sinh(ar_1)\sinh(ar_2)\cos(\alpha)-1}
  {\cosh(a\theta r_1)\cosh(a\theta r_2)-
  \sinh(a\theta r_1)\sinh(a\theta r_2)\cos(\alpha)-1}.
  $$
  
  An important point in the proof of lemma \ref{lem_triangle_upper} is that,
  whenever $\theta\le 1$, the function $F_a$ is nondecreasing with respect to
  $\alpha$.
  \label{rem_triangle}
\end{rem}
\begin{proof}[Proof of lemma \ref{lem_triangle_upper}]
  First consider a comparison triangle $(\bar{x}\bar{y}\bar{z})$
  in $\M^2(-a^2)$, that is such that   $d(\bar{x},\bar{y})= r_1$,
  $d(\bar{x},\bar{z})=r_2$, and $d(\bar{y},\bar{z})\!=\!d(y,z)$.
  Define $\bar{p}$, $\bar{q}$ to be the points on the geodesic segments
  $\bar{x}\bar{y}$ and $\bar{x}\bar{z}$ respectively, such that
  $d(\bar{x},\bar{p})=\theta r_1$, 
  $d(\bar{x},\bar{q})=\theta r_2$, and let
  $\bar{\alpha}=\angle_{\bar{x}}(\bar{y},\bar{z})$.
  By Toponogov theorem, we have $d(\bar{p},\bar{q})\ge d(p,q)$
  and $\bar{\alpha}\ge\alpha$.
  \begin{figure}[htbp]
    \scriptsize
    $$
    \input{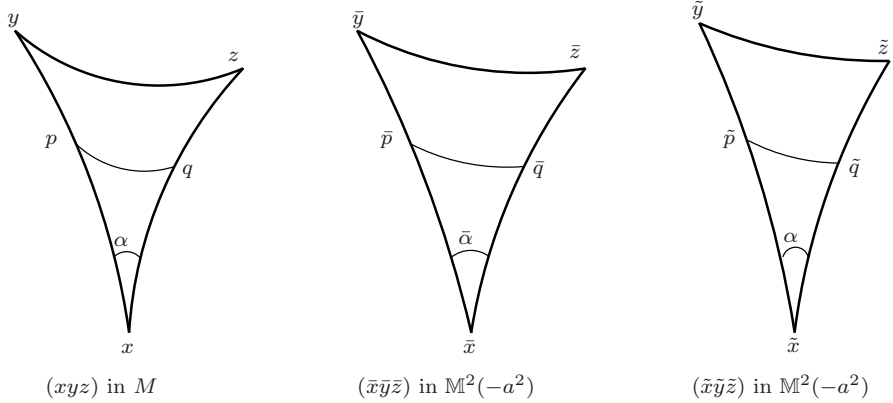}
    $$
    \caption{comparison triangles}
    \normalsize
    \label{fig-triangles}
  \end{figure}
  Using these inequalities and remark \ref{rem_triangle} we have
  $$
  \frac{C_a(d(y,z))}{C_a(d(p,q))} \ge
	\frac{C_a(d(y,z))}{C_a(d(\bar{p},\bar{q}))}
	= F_a(r_1,r_2,\bar{\alpha},\theta)
	\ge F_a(r_1,r_2,\alpha,\theta)
	= \frac{C_a(d(\tilde{y},\tilde{z}))}{C_a(d(\tilde{p},\tilde{q}))}
  $$
\end{proof}

A similar inequality holds for CH-manifolds with  curvature lower bound:

\begin{lem}[Triangle comparison with   curvature lower bound] 
  Let $M$ be a CH-manifold with $K_M\ge-b^2$.
  Let $(xyz)$ and $(\tilde{x}\tilde{y}\tilde{z})$ be triangles in $M$ and
  $\M^2(-b^2)$ respectively, such that
  $r_1=d(x,y)=d(\tilde{x},\tilde{y})$,  $r_2=d(x,z)=d(\tilde{x},\tilde{z})$
  and  $\alpha =\angle_x(y,z)=\angle_{\tilde{x}}(\tilde{y},\tilde{z})$.
  Moreover, for $\theta\in]0,1[$  let $p$, $q$ and $\tilde{p}$, $\tilde{q}$
  be respectively the points on the geodesic segments $xy$, $xz$ and
  $\tilde{x}\tilde{y}$, $\tilde{x}\tilde{z}$  such that
  $d(x,p)=d(\tilde{x},\tilde{p})=\theta r_1$ and 
  $d(x,q)=d(\tilde{x},\tilde{q})=\theta r_2$.
  Then: 
   $$
  \frac{C_b(d(y,z))}{C_b(d(p,q))} \le
  \frac{C_b(d(\tilde{y},\tilde{z}))}{C_b(d(\tilde{p},\tilde{q}))} =
  F_{b} (r_1,r_2, \alpha, \theta).
  $$
  \label{lem_triangle_lower}
\end{lem}

\begin{proof}
  The proof is similar to that of lemma \ref{lem_triangle_upper}. Toponogov
  theorem gives $d(\bar{p},\bar{q})\le d(p,q)$ and $\bar{\alpha}\le\alpha$,
  and by the monotonicity of the function $F_b$ we get
  $$
  \frac{C_b(d(y,z))}{C_b(d(p,q))} \le
	\frac{C_b(d(y,z))}{C_b(d(\bar{p},\bar{q}))}
	= F_b(r_1,r_2,\bar{\alpha},\theta)
	\le F_b(r_1,r_2,\alpha,\theta)
	= \frac{C_b(d(\tilde{y},\tilde{z}))}{C_b(d(\tilde{p},\tilde{q}))}
  $$
\end{proof}

\subsubsection*{Comparison of spheres}

Let $S_x(r)$ and $S_y(R)$ be two geodesic spheres in $M$, with $r<R$,  tangent 
at some point $z$,   with   $S_x(r)$ internal to  $S_y(R)$. Let $\vec \A_x$ and
$\vec \A_y$ (resp.  $\A_x, \A_y$) be the second, vector-valued (resp. scalar)
fundamental forms of $S_x(r)$ and $S_y(R)$, and let $\nu$ be the common inner
unit normal vector at $z$. We will now compare the two second fundamental forms
$\A_x$ and $\A_y$. 

Let $u\in T_zS_x(r)$ be a unitary vector, and let $c_u(s)$ be the geodesic of
$S_x(r)$ with initial tangent vector $u$. Denote by $r_x$ and $r_y$ the distance
functions to $x$ and $y$ respectively, and let $r_y(s)=r_y(c_u(s))$ be the
restriction of the function $r_y$ to the curve $c_u$. 
Applying proposition \ref{prop_hessian} to $c$ and $r_y$ we find
$$
r_y''(0) = \bigl(\hess^M r_y\bigr)(u,u) + \langle\nabla r_y,\vec \A_x(u,u)\rangle,
$$
and, since $\hess^M r_y$ gives the second fundamental form of $S_y(R)$ w.r. to  
$\nu$,  
\begin{equation}
  r_y''(0) =  \A_y(u,u)   - \A_x(u,u) 
  \label{f_second}
\end{equation}
But  $r_y''(0)\le 0$ as $z$ is the maximum of  $r_y$ on $S_x(r)$, thus  at the
point  $z$ we have $\A_y \le  \A_x$ which means that $S_x(r)$ is ``more curved''
than $S_y(R)$. Using the above comparison lemmas for triangles, we get sharper
comparison estimates for the tangent spheres :

\begin{lem}
  Let $(M,g)$ be a CH-manifold with  $K_M \le -a^2$. With the above notations,
  the second fundamental forms of $S_x(r), S_y(R)$ at the tangent point $z$
  satisfy:
  $$ 0 \le  \A_x - \A_y  \le (\cot_a r -\cot_a R) \, g $$
  Moreover, if we assume $-b^2 \le K_M$  then at the tangent point $z$  we  also
  have: 
  $$ (\cot_b r -\cot_b R ) \, g   \le  \A_x  -   \A_y $$ 
  \label{lem_spheres}
\end{lem}
 
\begin{rem}
  These estimates are optimal, since they are  equalities when  $M$  has,
  respectively, constant curvature $-a^2$ or $-b^2$.
\end{rem}

\begin{proof}
	We only consider the case $a>0$ ; when $a=0$, the proof is similar (just
	replace the hyperbolic laws by the Euclidean ones) and is left to the reader.

	As before, let $u\in T_zS_x(r)$ be a unitary vector,  let  $c(s)$ be the
	geodesic of $S_x(r)$ with initial tangent vector $u$ and let $r_y(s)$ be the
	restriction of the function $r_y$ to  the curve $c$. 
    For $s>0$ we consider (cf. figure~\ref{fig1}) :
    \begin{itemize}
        \item the angle  $\alpha(s)$ between $\nabla r_y$ and $\nabla r_x$ at
        	$c(s)$;
        \item the angle  $\beta(s)$  between the geodesic lines from $y$ to $z$
            and from $y$ to $c(s)$;
        \item $\theta=\frac{R-r}{R}$ and  the point $x(s)$ of the geodesic
        	from $y$ to $c(s)$ such that $d(y,x(s))=\theta d(y,c(s))$;
    \end{itemize}
	so  $r_y'(s)=\langle\nabla r_y ,\dot{c}(s)\rangle   =-\sin(\alpha(s))$.
    \begin{figure}[htbp]
        \scriptsize
        $$
        \input{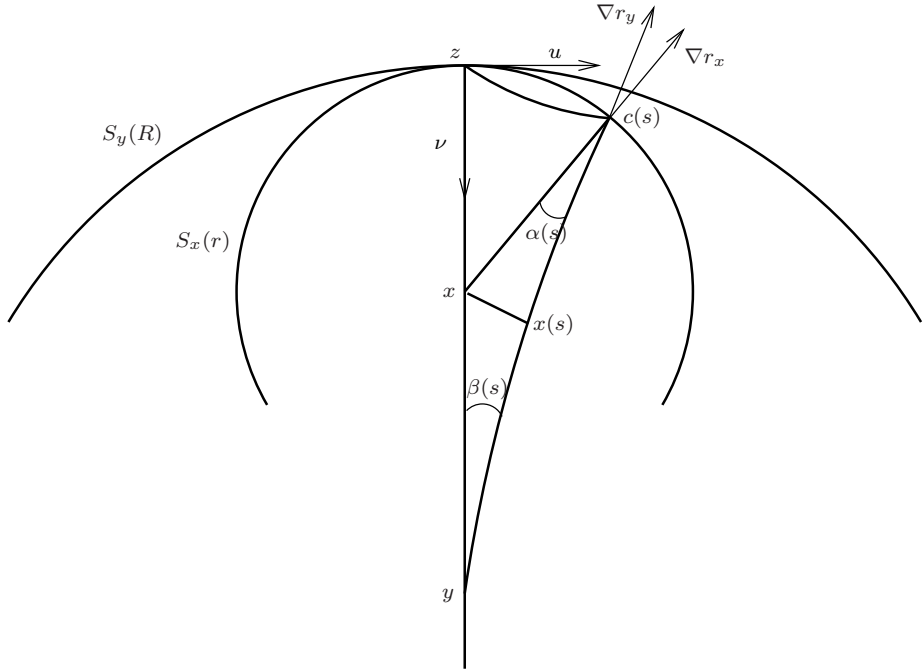}
        $$
        \caption{Comparing    the second fundamental forms of tangent spheres}
        \normalsize
        \label{fig1}
    \end{figure}
	Using Toponogov theorem for the triangle $(c(s)x(s)x)$ and the law of cosine  
	in $\M^2(-a^2)$ we get
	\begin{eqnarray}
		\cosh(ad(x,x(s))) & \ge & \cosh(ar)\cosh(a(1-\theta)r_y(s)) \nonumber \\
		 & & - \sinh(ar)\sinh(a(1-\theta)r_y(s))\cos(\alpha(s)) \nonumber \\
		 & \ge & 1 + \sinh(ar)\sinh(a(1-\theta)r_y(s))(1-\cos(\alpha(s))
		\label{eqn_carnot}
	\end{eqnarray}
	On the other hand, lemma \ref{lem_triangle_upper} applied to the triangle
    $(yzc(s))$ implies that
    \begin{equation}
      \cosh(ad(x,x(s)))-1 \le
      \frac{\cosh(ad(z,c(s)))-1}{F_a(R,r_y(s),\beta(s),\theta)}
      \label{eqn_use_lemma_a}
    \end{equation}
    which, plugged in (\ref{eqn_carnot}), yields  :   
	\begin{equation}
		1-\cos(\alpha(s)) \le
        	\frac{\cosh(ad(z,c(s)))-1}
        	{\sinh(ar)\sinh(a(1-\theta)r_y(s))F_a(R,r_y(s),\beta(s),\theta)}
      \label{eqn_cos}
    \end{equation}
	We divide  by $s^2$ and pass  to the limit for $s \to 0$ in
	(\ref{eqn_cos}) : as $r_y'(s)^2  = \sin^2 \alpha(s)$ and $r_y'(0)=0$,
	we have $\lim_{s \to 0}\frac{1 - \cos \alpha(s)}{s^2} = 
	\frac12  \lim_{s \to 0} \left( \frac{r_y'(s)}{s} \right)^2 =
	\frac12  r_y''(0) ^2$; 
	then,  notice that $\frac{d(z,c_u(s))}{s}\to 1$ and that, as
	$r_y(s)-R = O(s^2)$ and $\beta(s) = O(s)$, we have
	$\lim_{s\to 0}F_a(R,r_y(s),\beta(s),\theta)=
	\frac{\sinh^2(aR)}{\sinh^2(a(R-r))}$. So from (\ref{eqn_cos}) we get
	$$
	\left| r_y''(0)\right| \le \frac{a\sinh(a(R-r))}{\sinh(ar)\sinh(aR)} =
	a(\coth(ar)-\coth(aR))
	$$
	By (\ref{f_second}), as  $r_y''(0)\le 0$ we  deduce 
	$\A_x(u,u) - \A_y(u,u) \le \cot_a r - \cot_a R$.
   
	Consider now the curvature lower bound $-b^2\le K_M$. By Toponogov theorem
    and the law of cosine, equation (\ref{eqn_carnot}) becomes
    \begin{equation}
		1-\cos(\alpha(s)) \ge
		  \frac{\cosh(bd(x,x(s)))-1}{\sinh(br)\sinh(b(1-\theta)r_y(s))}
      	+ \frac{1-\cosh(br(1-\frac{r_y(s)}{R}))}{\sinh(br)\sinh(b(1-\theta)r_y(s))}
      \label{eqn_carnot-b}
    \end{equation}
	while lemma \ref{lem_triangle_lower} implies
    \begin{equation}
      \cosh(bd(x,x(s)))-1 \ge
      \frac{\cosh(bd(z,c(s)))-1}{F_b(R,r_y(s),\beta(s),\theta)},
    \end{equation}
    which plugged in (\ref{eqn_carnot-b}) yields
    \begin{multline}
      1-\cos(\alpha(s)) \ge
        \frac{\cosh(bd(z,c(s)))-1}
        {\sinh(br)\sinh(b(1-\theta)r_y(s))F_b(R,r_y(s),\beta(s),\theta)} \\
        + \frac{1-\cosh(br(1-\frac{r_y(s)}{R}))}{\sinh(br)\sinh(b(1-\theta)r_y(s))} 
      \label{eqn_cos-b}
    \end{multline}
    Dividing by $s^2$ and letting $s \to 0$  as before, we get
    $ (\cot_b r -\cot_b R) g   \! \le\!  \A_x  \! - \!  \A_y  $.
\end{proof}

In the sequel,  we will be mainly interested in the second fundamental form of
horospheres. We will use a result similar to lemma \ref{lem_spheres}, where the
sphere $S_y(R)$ is replaced by a horosphere:
\begin{lem}
  Let $(M,g)$ be a CH-manifold with  $K_M \le -a^2$. \\
  Let $S_x(r)$ and $H_{\xi}(z)$ be  respectively  a sphere and a horosphere  tangent 
  at a point $z$, with   $S_x(r)$ internal to the horosphere.
  Let $\A_x, \A_{\xi}$  be the second fundamental forms of $S_x(r), H_{\xi}(z)$  with
  respect to  be the common inner unit normal vector at \nolinebreak $z$.
  Then, at the tangent point $z$ we have  : 
  \begin{equation}
  	 0 \le \A_x - \A_{\xi}  \le  ( \cot_a r - a  ) \, g
  	\label{eqfundamentalforms1}
  \end{equation}
   Moreover, if we assume $-b^2 \le K_M$,  then at the tangent point $z$  we  also
   have: 
   $$   ( \cot_b r -b  ) \, g \le  \A_x  -  \A_\xi $$
  \label{lem_horosphere}
\end{lem}
This result can be obtained in two different ways: taking limits, in the inequalities
of lemma \ref{lem_spheres}, for $y$ tending to  $\xi$ along the geodesic
$\exp_z(t\nu)$, or following the same proof  with the Busemann function $b_\xi$ in
place of $r_y$. The proof is left to the reader.

%
%
%
%
\section{Asymptotically harmonic CH-manifolds}

In this section,  $M$ will always be an asymptotically harmonic CH-manifold with
horospheres of constant mean curvature $h$. 

\subsection{The entropy and the spectrum}
%
We are interested here in two invariants of the manifold $M$ : the volume entropy and
the spectrum. The entropy is determined by the behaviour of the volume of balls
whose second derivative (with respect to the radius) is given, in turns, by the mean
curvature of the spheres.  
On the other hand, the spectrum can be determined by using special functions whose
Laplacian has a nice behaviour; in our case, the distance function, whose Laplacian
is again given by the mean curvature of spheres (see discussion in \S1).

For points $x$ and $y$ in $M$, let $\vec{h}_x(y)$ be the mean curvature vector at $y$
of the sphere $S_x(d(x,y))$, and  $h_x(y)=-\langle \vec{h}_x(y),\nabla r_x \rangle$.
Notice that, as  $K_M\le 0$, balls and horoballs are convex, so both $h$ and 
$h_x(y)$ are non-negative.

\begin{lem}
    Let $M^{n+1}$ be an asymptotically harmonic CH-manifold.
    
    \noindent For all $x\in M$ and $r>0$, the sphere $S_x(r)$ satisfies
    $$
    \forall\ y\in S_x(r)\ \ \ h\le h_x(y) \le h + \frac{1}{r}.
    $$
    \label{main}
\end{lem}
\begin{proof}
  From lemma \ref{lem_horosphere} we have
  $$
  \A_{\xi} (u,u) \le \A_x (u,u) \le \A_{\xi} (u,u) + \frac{1}{r}|u|^2
  $$
  where $\A_x$ and $\A_{\xi}$ are, respectively,  the second fundamental forms of
  $S_x(r)$ and   of the horosphere  $H_{\xi}(y)$, tangent to $S_x(r)$ at $y$.
  Taking the trace on an orthonormal basis gives the result.
\end{proof}

We fix $x\in M$. For $r>0$,  let $B_x(r)$ be the ball of radius $r$ centered in $x$, 
and $V(r)=\Vol(B_x(r))$ the growth  function. 
The {\em entropy} of $M$ is defined by
$$
E = \limsup_{r\to\infty}\frac{1}{r}\log V(r).
$$

A first consequence of asymptotic harmonicity is the following linear
isoperimetric inequality :
\begin{prop}
    Let $M^{n+1}$ be an asymptotically harmonic CH-manifold.
    
    \noindent   For any domain $\Omega  \! \subset \! M$ with smooth boundary
    $\partial\Omega$ we have $nh\Vol(\Omega) \! \le \! \vol(\partial\Omega)$.
    \label{isop}
\end{prop}
\begin{proof}
    Fix some $\xi\in\dM$. Since $-\Delta b_\xi = nh$, integrating by parts on
    $\Omega$ the function $-\Delta b_\xi$ gives the result.
\end{proof}

\begin{thm}\label{thm_entropy}
    Let $M^{n+1}$ be an asymptotically harmonic CH-manifold.
    
    \noindent The entropy of $M$ is $E=nh$.
\end{thm}
\begin{proof}
    By the co-area formula we have $V'(r)=\vol(S_x(r))$, and by
    proposition~\ref{isop} we get $nhV(r)\le V'(r)$. Integrating this
    inequality we get $V(r) \ge A\ex^{nhr}$ for some constant $A$, so
    that the entropy is bounded below by $nh$.

    Now, the second derivative of $V$ is given by
    $V''(r)=n\int_{S_x(r)}h_x(y)dv_r(y)$
    where $dv_r$ is the volume form of $S_x(r)$.
    Choose $\varepsilon>0$ and let $r_0=\frac{1}{\varepsilon}$. By
    lemma~\ref{main}, we have $V''(r)\le n(h+\varepsilon)V'(r)$
    for any $r\ge r_0$. Integrating this inequality between $r_0$
    and $r$, yields $V'(r)\le A\ex^{n(h+\varepsilon)r}$ for
    some constant $A$. Integrating once again between $r_0$
    and $r$, we get $V(r)\le B + C\ex^{n(h+\varepsilon)r}$,
    which implies that $E\le n(h+\varepsilon)$.
    Since $\varepsilon$ is arbitrarily small, this concludes the proof.
\end{proof}

\begin{thm}\label{thm_spectrum}
    Let $M^{n+1}$ be an asymptotically harmonic CH-manifold.
    
    \noindent The spectrum of the Laplacian of $M$ is
    $\sigma(\Delta) = [\frac{n^2h^2}{4},+\infty)$
\end{thm}
\begin{proof}
    By proposition \ref{isop} and Cheeger's inequality, we have
    $\sigma(\Delta)\subset[\frac{n^2h^2}{4},+\infty)$.

    Conversely, we choose $x\in M$ and consider 
    the distance function $r_x$ to $x$. Since the Laplacian of $r_x$ is
    given by the mean curvature of spheres, we have
    \begin{equation} \label{eqdonnelly}
    	\sup_{y\in M \! \setminus \! B_x(R) } \, \{ \, |\Delta r_x(y)-nh|\  \} \;
    	\le \; \frac{n}{R}
    \end{equation}
    Using (\ref{eqdonnelly}) and the fact that   $|\nabla r_x|=1$, we can follow the
    method initiated by H. Donnelly
    to determine the essential spectrum (cf. \cite{Donnelly}) : for each
    $\lambda>\frac{n^2h^2}{4}$ we use radial functions to construct sequences
    satisfying Weyl's criterion for $\lambda$ (cf. \cite{Reed-Simon}
    theorem VII.12 p. 237).
    See for example \cite{Kumura} theorem 1.2
    for a general result, whose hypotheses are satisfied by the function $r_x$.
\end{proof}
\begin{rem}
	From theorems \ref{thm_entropy} and \ref{thm_spectrum} we deduce
	$\inf \{ \sigma(\Delta) \}=\frac{E^2}{4}$. 
	
	For cocompact negatively curved manifolds, this equality is  equivalent to the
	asymptotic harmonicity (cf. \cite{Ledrappier} theorem 1).
	But, in the general case, it is easy to construct manifolds satisfying this
	inequality, which are not asymptotically harmonic. For example, the conclusions
	of theorems \ref{thm_entropy} and \ref{thm_spectrum} hold true for any
	Cartan-Hadamard manifold with curvature less than $-h^2$ and tending to
	$-h^2$ at infinity.
\end{rem}

\subsection{Rigidity}

Consider the second fundamental form $\A_{\xi}$ of a horosphere $H_{\xi}$,  and let
$\lambda_1,\dots,\lambda_n$ be the principal curvatures
of $H_{\xi}$ at some point $x$, with respect to the inner unit normal of $H_{\xi}$.
If  $M$ satisfies the curvature upper bound $K_M\le-a^2$, then it is well known that
$\lambda_i\ge a$ (cf \cite{Karcher}). Therefore we get
$$
n^2h^2 = (\sum_i\lambda_i)^2 = \sum_i\lambda_i^2 + 2\sum_{i<j}\lambda_i\lambda_j
	 \ge |\A_{\xi}|^2 + n(n-1)a^2,
$$
and
\begin{equation} \label{eqn_bound-A}
	|\A_{\xi}|^2 \le n^2h^2 - n(n-1)a^2.
\end{equation}
When assuming a curvature lower bound $K_M\ge-b^2$, a similar argument gives
\begin{equation}
	|\A_{\xi}|^2 \ge n^2h^2 - n(n-1)b^2.
\end{equation}

Now, as the mean curvature is the same for all horospheres, taking the
trace of Riccati equation (\ref{eqn_Riccati}) gives $|\A_{\xi}|^2 + \Ric_M(u,u)=0$
for any $u\in SM$, for the second fundamental form $\A_{\xi}$   of a horosphere
tangent to $u^\bot$. Therefore we get :
\begin{prop} \label{prop_bound-Ric}
	Let $M^{n+1}$ be an asymptotically harmonic CH-manifold. For any
	$u\in SM$ we have
	\begin{enumerate}
		\item if $M$ satisfies  $K_M\le-a^2$, then
			$\Ric_M(u,u) \ge -n^2h^2 + n(n-1)a^2$;
		\item if $M$ satisfies  $K_M\ge-b^2$, then
			$\Ric_M(u,u) \le -n^2h^2 + n(n-1)b^2$.
	\end{enumerate}	
\end{prop}

As a consequence, we have the following characterization of constant
curvature spaces :
\begin{cor} \label{cor_rigidity}
    Let $M^{n+1}$ be an asymptotically harmonic CH-manifold.
	\begin{enumerate}
		\item if $M$ satisfies  $K_M\le-a^2$ then $h\ge a$, and
			$h=a$ if and only if $M=\M^{n+1}(-a^2)$;
		\item if $M$ satisfies   $K_M\ge-b^2$ then $h\le b$, and
			$h=b$ if and only if  $M=\M^{n+1}(-b^2)$.
	\end{enumerate}	
\end{cor}
\begin{proof}
    The curvature upper bound $K_M\le-a^2$ implies $h\ge a$.
    If $h=a$, then proposition \ref{prop_bound-Ric} gives $\Ric_M\ge-na^2$, and
    since the Ricci curvature is a sum of $n$ sectional curvatures which are not
    greater then $-a^2$, this implies that all the sectional curvatures are
    equal to $-a^2$.  The proof is the same when assuming a curvature lower bound.
\end{proof}

\subsection{Growth of horospheres}

It is well known that, on CH-manifolds with {\em pinched} curvature, horospheres
have polynomial volume growth, whose degree depend on the bounds on
the curvature (cf. \cite{Karp-Peyerimhoff}). We will now see that, under the
asymptotic harmonicity assumption, an upper bound   $K_M \le -a^2 < 0$ is enough to
estimate from above the polynomial growth of horospheres.

Let $H_{\xi}$ be a horosphere centered in some point at infinity $\xi$, let $b_\xi$
be the Busemann function vanishing on $H_{\xi}$, and let $g_0$ be the Riemannian
metric induced on $H_{\xi}$.   For each $t\in\R$, there is a natural diffeomorphism
$\varphi_t:H_\xi \to H_\xi(t)$ defined by $\varphi_t(x)=\exp_x(t\nabla b_\xi)$, which
in turns induces a diffeomorphism
$$
\Phi\left\{\begin{array}{rcl}
                \R\times H_\xi(0) & \to & M \\
                (t,x) & \mapsto & \varphi_t(x)
           \end{array}\right.
$$
In these ``horospherical'' coordinates $(t,x)$, the metric of $M$ reads
$g = dt^2 + g_t $, where $g_t = \varphi_t^* g_{H_\xi(t)} $ and $g_{H_\xi(t)}$ is the
induced Riemannian metric of $H_{\xi} (t)$.
 
When assuming a sectional curvature upper bound $K\le-a^2$, the map $\varphi_t$  
increases the distance for $t>0$ and decreases the distance if $t<0$.
In fact, as a consequence of comparison theorem for Jacobi fields, we have that all
the eigenvalues of   $d\varphi_t$ are greater than or equal to $\ex^{at}$ if $t>0$,
and less than or equal to $\ex^{at}$ if $t<0$ (cf. \cite{Heintze-Hof}).

Now, it is a standard fact that the mean curvature gives the derivative of the volume
form of a submanifold under a deformation. 
In our setting, if   $dv_t = J_t(x)dv_0$ is the volume form of the metric $g_t$ and
$J_t(x)$ is the density of $dv_t$ with respect to $dv_0$, we have 
$J_t ' =nh_t J_t$, where $h_t$ is the mean curvature of $H_\xi(t)$. 
By asymptotic harmonicity, we deduce that $dv_t=\ex^{nht}dv_0$ for all $t$;
therefore, in horocyclic coordinates  the volume form of $M$ reads
$dv_M=\ex^{nht}dtdv_0$.

On the other hand, by theorem \ref{thm_entropy}, the volume entropy of $M$ is $nh$:
heuristically,  this means that the exponential rate of the volume growth of $M$
comes from the behaviour of the volume form in the $\R$ direction, and that the
volume growth of the slices $H_\xi(t)$ should be subexponential. Namely:

\begin{thm} \label{thm_growth}
	Let $M^{n+1}$ be an asymptotically harmonic CH-manifold with sectional
	curvature upper bound $K_M\le-a^2<0$. Then, there exists a constant $C$
	(depending only on $n$, $a$ and $h$) such that,
    for any horosphere $H$ of $M$, the balls of $H$ satisfy
    $\vol(B^H_x(r))\le Cr^{\frac{nh}{a}}$ for all $r>0$.
\end{thm}
\begin{proof}
	Let $H=H_\xi$ be a horosphere centered in $\xi$. For any $u,v \in TH$, Gauss
	equation implies that
	$K_{H}(u,v)=K_M(u,v) + \A_{\xi} (u,u)\A_{\xi} (v,v)-\A_{\xi} (u,v)^2$, where
	$K_{H}$ and $K_M$ are the sectional curvatures of $H$ and $M$ respectively, and
	$\A_{\xi}$ is the second fundamental form of $H$. Taking the trace with respect
	to $v$ we get
	\begin{eqnarray*}
		\Ric_H(u,u) & = & \Ric_M(u,u)-K_M(u,\nu)+nh\A_{\xi}(u,u)-
			\sum_i\A_{\xi}(u,e_i)^2 \\
		 & \ge & \Ric_M(u,u) - |\A_{\xi}|^2  \ge  -2n^2h^2 + 2n(n-1)a^2
	\end{eqnarray*}
	where the last inequality comes from (\ref{eqn_bound-A}) and
	Proposition \ref{prop_bound-Ric}.
	Therefore, by Bishop's comparison theorem,
	there exists a constant $C$ (depending only on $n$, $a$ and $h$) such that, for
	any $x$ in  $H$ we have $\Vol(B^H_x(1))\le C$.
	
	Let now $x\in H$  and consider the map $\varphi_{-t}:H\to H_\xi(-t)$ defined
	above, for $t>0$. 
	As $K_M\le-a^2$, we have
	$\varphi_{-t}(B^H_x(r))\subset B^{H_{\xi}(-t)}_{\varphi_{-t}(x)}(\ex^{-at}r)$.
	Moreover, as $dv_{-t}=\ex^{-nht}dv_0$, we have
	$\vol(\varphi_{-t}(B^H_x(r)))=\ex^{-nht}\vol(B^H_x(r))$; so, choosing
	$t=\frac{\ln r}{a}$ we obtain
	$$
	\vol(B^H_x(r)))
		\le \ex^{nh\frac{\ln r}{a}}\vol(B^{H_{\xi}(-t)}_{\varphi_{-t}(x)}(1))
		\le Cr^{\frac{nh}{a}}
	$$
\end{proof}

\begin{rem} \label{rem_optimal}
	This theorem proves that the degree of the polynomial volume
	growth of  the horospheres is bounded above by $\frac{nh}{a}$. This
	upper bound is sharp, as it is the degree of the volume growth
	of the horospheres in the hyperbolic space (the horospheres being
	Euclidean in that case). Note that the upper bound is
	also sharp for the rank one symmetric spaces.
\end{rem}
\begin{rem} \label{rem_lower}
	Using a similar proof, it is easy to see that the  lower bound
	$-b^2 \le K_M \le 0$  gives a lower bound on the volume growth of the
	horospheres, namely 
	$ \vol(B^H_x(r)) \ge Cr^{\frac{nh}{b}}$. The proof is left to the reader.
\end{rem}
%

\subsection{The mean value property}

Harmonic manifolds are characterized by the fact that the harmonic functions
have the mean value property : for any harmonic function $F$ and any $R>0$,
$$
F(x_0) = \frac{1}{\vol (S_{x_0}(R))} \int_{S_{x_0}(R)} F(x) dv_{S_{x_0}(R)}
$$
This can be proved by taking the derivative of the right-hand side of the above
equality, and by observing that it vanishes for any harmonic function $F$ if
and only if the spheres have constant mean curvature.

In the following theorem we prove that harmonic functions on an asymptotically
harmonic manifold satisfy a mean value property, where, naturally, the mean
is taken on horospheres. As the horospheres are non-compact, the mean on
an horosphere is obtained as the limit of the means on an exhaustion.
The computations of these horospherical means are very similar to those
in \cite{Karp-Peyerimhoff}.

\begin{thm} \label{thm_meanvalue}
	Let $M^{n+1}$ be an asymptotically harmonic manifold with
    sectional curvature upper bound $K_M\le-a^2<0$, and  
    let $F$ be a function which is continuous on $M\cup\dM$ and harmonic on $M$.
    
    For any $\xi\in\dM$, any horosphere $H_\xi$ centered in $\xi$, and
    any $x\in H_\xi$, there
    exists a sequence $(r_j)_{j\in\N}$ tending to $+\infty$ such that
    $$
    \lim_{j\to\infty}\frac{1}{\Vol(B_x^{H_\xi}(r_j))}
    	\int_{B_x^{H_\xi}(r_j)}Fdv_{H_\xi} = F(\xi)
    $$
    where $B_x^{H_\xi}(R)$ denote the ball in $H_\xi$
    centered in $x$ of radius $R$.
\end{thm}
\begin{proof}
	Let $H_{\xi}$ be a horosphere centered in some point at infinity $\xi$,
	and let $\varphi_t:H_\xi \to H_\xi(t)$ be the diffeomorphism defined in
	\S 3.3.
	
	Choose $x\in H_\xi$. Because $H_\xi$ has polynomial volume growth, there
	exists a sequence $(r_j)_{j\in\N}$ tending to $+\infty$ such that
	$$
	\lim_{j\to\infty}\frac{\vol(\partial B_x^{H_\xi}(r_j))}{\Vol(B_x^{H_\xi}(r_j))}=0.
	$$
	
	For $t\in\R$ and $j\in\N$, let $\Omega_{j,t}=\varphi_t(B_x^{H_\xi}(r_j))$.
	As pointed out in \S 3.3, we have
	$\Vol(\Omega_{j,t})=\ex^{nht}\Vol(B_x^{H_\xi}(r_j))$. Moreover, the boundary
	of $\Omega_{j,t}$ satisfy
	$$
	\frac{d}{dt}\vol(\partial\Omega_{j,t}) = -(n-1)\int_{\partial\Omega_{j,t}}
		\langle\vec{k}_{j,t},\ddt\rangle
	$$
	where $\vec{k}_{j,t}$ is the mean curvature vector of $\partial\Omega_{j,t}$
	(seen as a submanifold of $M$).
	Taking an orthonormal basis $(e_1,\dots,e_{n-1})$ of $T\partial\Omega_{j,t}$
	and $\eta_{j,t}$ its exterior unit normal in $H_\xi(t)$ we have
	$$
	-(n-1)\langle\vec{k}_{j,t},\ddt\rangle =
		\sum_{i=1}^{n-1}\langle D^M_{e_i}\ddt,e_i\rangle
		= nh - \langle D^M_{\eta_{j,t}}\ddt,\eta_{j,t}\rangle
		\le nh - a
	$$
	where the last inequality comes from the curvature upper-bound on $M$.
	Therefore we have
	$\frac{d}{dt}\vol(\partial\Omega_{j,t})\le(nh-a)\vol(\partial\Omega_{j,t})$,
	and integrating this inequality we get
	$\vol(\partial\Omega_{j,t})\le\ex^{(nh-a)t}\vol(\partial\Omega_{j,0})$ and
	\begin{equation}\label{eqn_volume-quotient}
		\frac{\vol(\partial\Omega_{j,t})}{\Vol(\Omega_{j,t})} \le
			\ex^{-at}\frac{\vol(\partial B_x^{H_\xi}(r_j))}{\Vol(B_x^{H_\xi}(r_j))}.
	\end{equation}
	
	Consider now
	\begin{equation}\label{eqn_def-g_j}
		g_j(t) = \frac{1}{\Vol(\Omega_{j,t})}\int_{\Omega_{j,t}}Fdv_t
	\end{equation}
	where $dv_t$ is the volume form of $H_\xi(t)$ and $F$ a function which is
	continuous on $M\cup\dM$ and harmonic on $M$. In particular, $F$ is bounded.
	Using the fact that horospheres have constant mean curvature, we have
	\begin{equation}\label{eqn_g_j-prime}
		g_j'(t) = \frac{1}{\Vol(\Omega_{j,t})}\int_{\Omega_{j,t}}
			\langle\nabla F,\ddt\rangle dv_t
	\end{equation}
	and
	\begin{equation}\label{eqn_g_j-second}
		g_j''(t) = \frac{1}{\Vol(\Omega_{j,t})}\int_{\Omega_{j,t}}.
			(\hess^MF)(\ddt,\ddt) dv_t
	\end{equation}
	Using proposition \ref{prop_hessian} and the fact that $F$ is harmonic in
	$M$ we get
	$$
	(\hess^MF)(\ddt,\ddt) = - \tr((\hess^{H_\xi(t)}F)_{|_{TH_\xi(t)}})
		= \Delta^{H_\xi(t)}f + nh\langle \nabla F,\ddt \rangle
	$$
	where $f$ is the restriction of $F$ to $H_\xi(t)$.
	Equation (\ref{eqn_g_j-second}) gives
	\begin{eqnarray}\label{eqn_edo-g_j}
		g_j''(t) - nhg_j'(t)  & =  & \frac{1}{\Vol(\Omega_{j,t})}\int_{\Omega_{j,t}}
			\Delta^{H_\xi(t)}f dv_t \nonumber \\
		 & = & -\frac{1}{\Vol(\Omega_{j,t})}\int_{\partial\Omega_{j,t}}
			\langle\nabla F,\eta_{j,t}\rangle dv_t.
	\end{eqnarray}
	
	As $\Ric_M$ is bounded from below and $F$ is bounded on $M$, using Yau's gradient
	estimate for harmonic functions \cite{Yau}, there exists a constant $C$ (depending
	on $n$, $a$, $h$ and $||F||_\infty$) such that $|\nabla F|\le C$ on $M$.
	Therefore, using (\ref{eqn_volume-quotient}), the right-hand side
	of (\ref{eqn_edo-g_j}) satisfies
	$$
	\left|\frac{1}{\Vol(\Omega_{j,t})}\int_{\partial\Omega_{j,t}}
			\langle\nabla F,\eta_{j,t}\rangle dv_t\right| \le
			C\ex^{-at}\frac{\vol(\partial B_x^{H_\xi}(r_j))}{\Vol(B_x^{H_\xi}(r_j))}
	$$
	and tends uniformly to zero on bounded intervals when $j$ tends to $+\infty$.
	In particular, it implies that, on bounded intervals, the $C^0$ norms of the
	functions $g_j''$ are uniformly bounded.
	The fact that $F$ is bounded and Yau's gradient estimate also imply that
	the $C^0$ norms of the functions $g_j$ and $g_j'$
	are uniformly bounded, and, using Arzela-Ascoli convergence theorem,
	we have that, up to a subsequence, $(g_j)_{j\in\N}$ tends in $C^1$ topology to
	a function $g$.
	
	Moreover, multiplying (\ref{eqn_edo-g_j}) by a test function, integrating by part
	and letting $j$ tend to $+\infty$ we find that, in the sense of distributions, $g$
	is a solution of
	$$
	g''(t) - nhg'(t) = 0.
	$$
	Therefore, by classical regularity theory, $g$ is smooth and
	$g'(t)=g'(0)\ex^{nht}$.
	Since $g'$ is bounded on $\R$ we must have $g'\equiv 0$ and $g$ is constant.
	
	For any neighbourhood $U$ of $\xi$ (for the cone topology) there exist
	$t$ such that the horosphere $H_\xi(t)$ is contained in $U$. By continuity
	of $F$ on $M\cup\dM$ and by the definition of $g_j$, the value $g_j(t)$ can
	be made arbitrary close to $F(\xi)$ (for any $j$). Therefore
	we have $g(t)=F(\xi)$ for any $t\in\R$, and $g(0)=F(\xi)$ gives the result.
\end{proof}
\begin{rem}
	It would be better to have a similar result without taking a sequence of
	radii tending to infinity, that is to have
	$$
	\lim_{r\to\infty}\frac{1}{\Vol(B_x^{H_\xi}(r))}
		\int_{B_x^{H_\xi}(r)}Fdv_{H_\xi} = F(\xi).
	$$
	For the proof to work in that case, one need to have
	$\displaystyle\lim_{r\to\infty}
	\frac{\vol(\partial B_x^{H_\xi}(r))}{\Vol(B_x^{H_\xi}(r))}=0$.
	However, from the polynomial volume growth of horospheres one only get
	$\displaystyle\liminf_{r\to\infty}
	\frac{\vol(\partial B_x^{H_\xi}(r))}{\Vol(B_x^{H_\xi}(r))}=0$.
\end{rem}

%
%
%
%

\section{Asymptotic behaviour of the volume form}

In the previous section, in order to compute the entropy, we integrated the
inequalities of lemma \ref{lem_horosphere} on spheres. But since these inequalities
hold pointwise, we can try to determine the asymptotic behaviour of the volume form
at least in a fixed direction. Actually, let $\theta_x (u,r)$ be the density of the
volume form of $M$ in normal coordinates centered in some point $x$; so the volume
form reads $dv_M=\theta_x(u,r)dv_{S_xM}dr$, where $dv_{S_xM}$ is the volume form
of $S_xM$.

Harmonic manifolds are characterized by the fact that $\theta_x(u,r)$ only
depends on $r$. In this section we give a characterization of asymptotically
harmonic manifolds in term of the asymptotic behaviour of $\theta_x(u,r)$:
\begin{thm} \label{thm_characterization}
	Let $M$ be a CH-manifold with  $K_M \le -a^2<0$ and entropy $E$. \\
	$M$ is asymptotically harmonic if and only if there exists a
	positive function $\tau:SM\to\R_+$ such that $\theta_x (u,r)$ is
	uniformly equivalent to $\tau (u) \ex^{Er}$ for $r \to \infty$.
\end{thm}
``Uniformly equivalent''  here means that the quotient of $\theta_x (u,r)$
by $\tau(u)\ex^{Er}$ converges to $1$ for $r\to \infty$, uniformly
with respect to $u \in SM$. This result will be consequence of the three
propositions proved in the following subsections.

A Riemannian manifold is harmonic if and only if the density function only
depends on $r$. As an asymptotic analogue, one would expect that
$\lim_{r\to\infty}\frac{\theta(u,r)}{\ex^{Er}}$ does not depend on $u$,
and thus that $\tau(u)$ be constant on $SM$. In proposition~\ref{prop_tau-properties}
we prove it holds under the restrictive assumption that $DR_M$ is bounded.

\subsection{The asymptotic volume-density function $\tau$}

The function $\theta_x$ is related to the mean curvature $h_x$ of spheres centered in
$x$ of radius $r$ by the formula 
\begin{equation} \label{eqcourburemoyenne}
	\frac{\theta_x'(u,r)}{\theta_x(u,r)}=nh_x(\exp_x(ru) )
\end{equation}
where  $\theta_x'$ denotes the derivative of $\theta_x$ with respect to $r$. 

In what follows,   we shall often write for short the point $\exp_x (ru)$ as $(u,r)$
to avoid cumbersome notations; 
moreover, we will regard  $\theta_x (u,r)$ as a function on $SM \times \R$, so we can 
drop the index $x$.

Using lemma \ref{lem_horosphere} we get the following result :
\begin{prop} \label{prop_volum-form}
	Let $M$ be a CH-manifold with curvature $K_M \le -a^2<0$. \\
	If $M$ is asymptotically harmonic, then there exists a bounded, positive function
	$\tau:SM\to\R_+$ such that
	$$
	\forall u\in SM\ \ \ \Bigl|\frac{\theta(u,r)}{\tau(u)\ex^{nhr}}-1\Bigr|
		\le \varepsilon(r)   
	$$
	for an explicit function  $\varepsilon(r)$  only depending on $a$ and $n$ such
	that $\lim_{r\to\infty}\varepsilon(r)=0$.
\end{prop}
\begin{proof}
  As $\frac{\theta^\prime(u,r)}{\theta(u,r)} = h_x (u,r)$, taking traces in
  (\ref{eqfundamentalforms1}) yields:
  $$
  0 \le \frac{\theta^\prime(u,r)}{\theta(u,r)} - nh
    \le na(\coth(ar)-1),
  $$
  that is
  $$
  0 \le \frac{d}{dr} \left[ \ln(\theta(u,r)\ex^{-nhr}) \right]
    \le na \left[ \coth(ar)-1\right].
  $$
  The first inequality implies that $\theta(u,r)\ex^{-nhr}$ is nondecreasing
  with respect to \nolinebreak $r$. On the other hand, integrating the second one
  gives
  \begin{equation} \label{eqn_theta-upper-bound}
    \theta(u,r)\ex^{-nhr} \le 
  	\theta(u,s)\ex^{-nhs}\ex^{na\int_s^r(\coth(at)-1)dt}
  \end{equation}
  and, as $\int_s^\infty(\coth(at)-1)dt$ is finite, we deduce that
  $\theta(u,r)\ex^{-nhr}$ is bounded, which implies that
  $\lim_{r\to\infty}\theta(u,r)\ex^{-nhr}$ exists.
  Therefore we can define the function $\tau$ on the unitary tangent bundle as
  \begin{equation}
	\tau(u) = \lim_{r\to\infty}\theta(u,r)\ex^{-nhr}
	\label{eqn_tau}
  \end{equation}
  Moreover, as $\theta(u,r)\ex^{-nhr}$ is nondecreasing we have
  \begin{equation}\label{eqn_tau-positive}
    \forall r>0\ \ \ 0< \theta(u,r)\ex^{-nhr} \le \tau(u)
  \end{equation}
  Again from   (\ref{eqn_theta-upper-bound}), subtracting $\theta(u,s)\ex^{-nhs}$ and
  letting $r \to \infty$ we deduce
  $$
  \tau(u) - \theta(u,s)\ex^{-nhs} \le
    \theta(u,s)\ex^{-nhs}\bigl(\ex^{na\int_s^\infty(\coth(at)-1)dt} -1\bigr)
  $$
  As $\theta(u,s)\ex^{-nhs}$ is nondecreasing  the left-hand side is nonnegative, so
  by (\ref{eqn_tau-positive}),
  $$
  \Bigl|\frac{\theta(u,r)}{\tau(u)\ex^{nhr}}-1\Bigr| \le
    \ex^{na\int_r^\infty(\coth(at)-1)dt} -1
  $$
  The right-hand side is uniformly bounded from above and tends to $0$ when $r$ tends
  to infinity, which concludes the proof.
\end{proof}

\subsection{Properties of the function $\tau$}

First, we remark that the function $\tau$ is bounded : for any
$u \in SM$ we have  $\tau(u) \le \frac{1}{(2a)^n}$. In fact,
from equation (\ref{eqn_theta-upper-bound}) we obtain
$$
\theta(u,r)\ex^{-nhr}
\le \frac{\theta(u,s)\ex^{-nhs}}{\sinh^n(as)}\ex^{-na(r-s)}\sinh^n(ar).
$$
Letting $s$ tend to $0$, as $\frac{\theta(u,s)}{s^n} \to 1$, we deduce
$$
\theta(u,r)\ex^{-nhr} \le \frac{\ex^{-nar}\sinh^n(ar)}{a^n},
$$
and for $r \to \infty$  we get $\tau(u) \le \frac{1}{(2a)^n}$.

Therefore, proposition \ref{prop_volum-form} implies
$$
|\theta(u,r)\ex^{-nhr} - \tau(u)| \le \frac{1}{(2a)^n}\varepsilon(r)
$$
and the function $\tau$ is the uniform limit of $\theta(u,r)\ex^{-nhr}$ ;
as the convergence is uniform,  the function $\tau$ is continuous on $SM$.
Moreover, {\em as soon as $\theta(u,r)\ex^{-nhr}$ has a limit}, this limit can be
expressed in terms of Jacobi tensors. This was used to study the asymptotic behaviour
of the volume on harmonic manifolds (cf. \cite{Knieper2,Connell,He-Kn-Sh}). Using
this approach we get more information on the function $\tau$.
\begin{prop} \label{prop_tau-properties}
	Let $M$ be an asymptotically harmonic CH-manifold with
	curvature $K_M \le -a^2<0$. Then:
	\begin{enumerate}
		\item $\tau: SM \to \R_+$ is invariant by the geodesic flow and flip
			invariant, i.e. :
			\begin{itemize}
				\item $\tau(\dot{\gamma}(t))$ is constant for any geodesic $\gamma$;  
				\item $\tau(v)=\tau(-v)$ for all $v\in SM$.
			\end{itemize}
		\item if $DR_M$ is bounded on $M$, then $\tau$ is constant on $SM$;
		\item $\tau \ge \frac{1}{(2h)^n}$, with equality if and only if the curvature
			is constant.
	\end{enumerate}
\end{prop}
\begin{proof}
	A Jacobi tensor  along a geodesic $\gamma$ is a smooth family $J(t)$ of
	endomorphisms of $\dot{\gamma}(t)^\bot$ satisfying the Jacobi equation
	$J''(t)+R(t)J(t)=0$, where $R(t)$ is defined from the Riemann tensor by
	$R(t)u=R(\dot{\gamma}(t),u)\dot{\gamma}(t)$. Then, applying $J$ to any parallel
	vector field $V(t)$ along $\gamma$  gives a  Jacobi vector field $J(t)V(t)$.
	
	Let $v\in S_xM$ and $\gamma(t)=\exp_x(tv)$, and consider the Jacobi tensor $J_v$
	along $\gamma$ defined by $J_v(0)=0$ and $J_v'(0)=\id$. It is well known that
	$J_v'(r)J_v^{-1}(r)$ gives the shape operator $A_{x} (v,r)$  of the sphere
	$S_x(r)$ at $\exp_x(rv)$ (with respect to the inner normal to the sphere), and
	that $\theta(v,r)=\det(J_v(r))$.
	
	For $r\!>\!0$, let  $U_{v,r}$, $S_{v,r}$ be the Jacobi tensors on $\gamma$
	defined by  $U_{v,r}(-r)\!=\!0$, $U_{v,r}(0)=S_{v,r}(0)=\id$  and $S_{v,r}(r)=0$.
	The unstable and stable Jacobi tensors at $v$ are defined by
	$U_v=\lim_{r\to\infty}U_{v,r}$ and $S_v=\lim_{r\to\infty}S_{v,r}$.
	As $U_{v,r}'(0)=J'_{\dot{\gamma}(-r)}(r)J^{-1}_{\dot{\gamma}(-r)}(r)$ is the
	shape operator  of the sphere $S_{\gamma(-r)}(r)$ at $x$, it follows that
	$U_v'(0)$ is shape operator at $x$ of the horosphere centered in
	$\xi_-=\lim_{r\to\infty}\gamma(-r)$. In a similar way, we have that $-S_v'(0)$ is
	the shape operator at $x$ of the horosphere centered in
	$\xi_+=\lim_{r\to\infty}\gamma(r)$.
	Since  $M$ is asymptotically harmonic, we have $\tr(U'_v(0))=nh$; following the
	proof of Corollary 2.5 of \cite{Knieper2} we get
	$$
	\theta(v,t)\ex^{-nht} = \frac{1}{\det(U'_v(0)-S'_{v,t}(0))},
	$$
	which, taking the limit for $t \to \infty$, gives
	\begin{equation} \label{eqtau}
		\tau(v) = \frac{1}{\det(U'_v(0)-S'_v(0))}.
	\end{equation}
	The proposition then follows from this expression of $\tau(v)$.
	
	First, as $U'_v(0)$ and $-S'_v(0)$ are the shape operators of the  horospheres
	centered in $\xi_-$ and $\xi_+$, relative to their respective inner normals, it
	is clear that $\tau$ is flip invariant.
	The invariance by the geodesic flow is just lemma 2.2 in \cite{He-Kn-Sh}.
	 
	To prove the second point, let us first show that $\tau (u)= \tau (v)$ when
	$u, v \in SM$ point towards the same boundary point $\xi\in\dM$, i.e.
	$\lim_{s\to+\infty}\gamma_u(s)=\lim_{s\to+\infty}\gamma_v(s)$. 
	By the invariance of $\tau$ under the geodesic flow, we may as well assume that
	$u$ and $v$ are normal to the same horosphere, so
	$d(\gamma_u(t),\gamma_v(t))\le c_1\ex^{-at}$ for all
	$t>0$. For any $r,t>0$  we have
	\begin{multline*}
		|\tau(u)-\tau(v)| \le
			|\tau(u)-\theta(\dot{\gamma}_u (t),r)\ex^{-nhr}|
			+ |\theta(\dot{\gamma}_u (t),r)
			- \theta(\dot{\gamma}_v (t),r)|\ex^{-nhr} \\
		+ |\tau(v)-\theta(\dot{\gamma}_v (t),r)\ex^{-nhr}|,
	\end{multline*}
	and using the invariance of $\tau$ by the geodesic flow and
	Proposition \ref{prop_volum-form} we get
	\begin{equation}\label{eqn_tau-tau}
		|\tau(u)-\tau(v)| \le (\tau(u)+\tau(v))\varepsilon(r)
			+ |\theta(\dot{\gamma}_u (t),r)
			- \theta(\dot{\gamma}_v (t),r)|\ex^{-nhr}.
	\end{equation}
	For $s\in]0,r]$, let $h_{u,t}(s)$ (resp. $h_{v,t}(s)$) be the mean curvature, at
	the point $\gamma_u (t+s)$ (resp.  at
	$\gamma_v (t+s)$), of the sphere of radius $s$ centered in
	$\gamma_u(t)$ (resp.   $\gamma_v(t)$).
	Following the Lemma 2.3 in \cite{He-Kn-Sh}, we will use comparison theory for
	Riccati equation to estimate	$|h_{u,t}(s)-h_{v,t}(s)|$.
	We choose  orthonormal parallel basis $e_{u,i}(s)$ of
	$\dot{\gamma}_u(t+s)^\bot$ and  $e_{v,i}(s)$ of
	$\dot{\gamma}_v(t+s)^\bot$) such that, for any $i$,
	$d(e_{u,i}(s),e_{v,i}(s))\le c_2\ex^{-a(t+s)}$ in $SM$,
	for some constant $c_2$.
	Let $A_{u,t}(s)$ and $A_{v,t}(s)$ be the matrices of the second fundamental forms
	of the spheres of radius $s$ centered in
	$\gamma_u(t), \gamma_v(t)$ in these basis.
	They  satisfy the Riccati equations $A'_{u,t}(s)+A^2_{u,t}(s)+R_{u,t}(s)=0$ and
	$A'_{v,t}(s)+A^2_{v,t}(s)+R_{v,t}(s)=0$, where $R_{u,t}(s)$  is the matrix of the
	endomorphism
	$R(\dot{\gamma}_u(t+s),.)\dot{\gamma}_u(t+s)$, 
	and analogously for  $R_{v,t}(s)$. Because of the assumption on $DR_M$, we have
	that the tensor  $r(s)=R_{u,t}(s)-R_{v,t}(s)$ satisfies
	\begin{equation}\label{eqn_r(t)}
		|r(s)|\le C_3\ex^{-a(t+s)}.
	\end{equation}
	
	Consider now   $B(s)=A_{u,t}(s)-A_{v,t}(s)$ and
	$Q(s)=\frac{1}{2}(A_{u,t}(s)+A_{v,t}(s))$. From the Riccati equations we have
	that $B$ is solution of
	$$
	B'(s) + B(s)Q(s) + Q(s)B(s) + r(s)=0.
	$$
	A direct computation shows that for any $0<\varepsilon<s$ we have the formula 
	\begin{equation}\label{eqn_B(s)}
		B(s)=  \,^tC(s)
			\left[\,^tC(\varepsilon)^{-1} B(\varepsilon) C(\varepsilon)^{-1}
			- \int^s_{\varepsilon} \!\!
			\,^t C(\zeta)^{-1} r(\zeta) C(\zeta)^{-1} d\zeta \right] C(s)
	\end{equation}
	where $C(s)$ is a solution of $C'(s)=-C(s)Q(s)$. In particular, because of the
	curvature upper bound we have $Q(s)\ge a\id$ hence, for any $0<\varepsilon<s$,
	$|C(\varepsilon)^{-1}C(s)|\le \ex^{-a(s-\varepsilon)}$. 
	Plugging  this estimate and (\ref{eqn_r(t)}) in the formula (\ref{eqn_B(s)}) we
	get
	$$
	|B(s)| \le |B(\varepsilon)|\ex^{-2a(s-\varepsilon)}
	 + c_4\ex^{-a(t+s)}
	$$
	Since both $A_{u,t}(s)$ and $A_{v,t}(s)$ behave, for $s \to 0$, as
	$\frac{1}{s}\id+o(1)$ we have $\lim_{\varepsilon\to 0}B(\varepsilon)=0$;
	therefore we deduce that $|B(s)|\le c_4 \ex^{-a(t+s)}$ and, taking the
	trace,
	$$
	|h_{u,t}(s)-h_{v,t}(s)| \le c_5\ex^{-a(t+s)}
	$$
	for some constant $c_5$. By the expression (\ref{eqcourburemoyenne}) for $h_x$,
	integrating on $[0,r]$ yields
	$$
	-c_6(1-\ex^{-ar})\ex^{-at} \le
		\ln\frac{\theta(\dot{\gamma_u}(t),r)}%
		{\theta(\dot{\gamma_v}(t),r)}
		\le c_6(1-\ex^{-ar})\ex^{-at}
	$$
	With these inequalities we can bound the last term of (\ref{eqn_tau-tau}):
	\begin{multline}\label{eqn_2nd-term}
		\theta(\dot{\gamma_v}(t),r)\ex^{-nhr}
			\bigl[\exp(-c_6(1-\ex^{-ar})\ex^{-at})-1\bigr] \\
		\le |\theta(\dot{\gamma_u}(t),r)-
			\theta(\dot{\gamma_v}(t),r)|\ex^{-nhr} \\
		\le \theta(\dot{\gamma_v}(t),r)
			\ex^{-nhr}\bigl[\exp(c_6(1-\ex^{-ar})\ex^{-at})-1\bigr]
	\end{multline}
	Choosing $r$ large enough, the term $\varepsilon(r)$ in (\ref{eqn_tau-tau}) can
	be made arbitrary small; for this value of $r$, $\theta(\dot{\gamma}_v(t),r)\ex^{-nhr}$
	stays close to $\tau (v)$ for all $t$ by proposition~\ref{prop_volum-form},
	and the above estimate (\ref{eqn_2nd-term}) implies that we can choose
	$t$ large enough to make also the last term of (\ref{eqn_tau-tau})  arbitrary
	small. Therefore $\tau(u)=\tau(v)$.

	Consider now any vector $u,v\in SM$, and let $\sigma$ be a geodesic such that
	$\lim_{s\to-\infty}\sigma(s) = \lim_{s\to+\infty}\gamma_u(s)$
	and $\lim_{s\to+\infty}\sigma(s) = \lim_{s\to+\infty}\gamma_v(s)$.
	From the  above computations we must have
	$\tau(v)=\tau(\dot{\sigma}(0))$ and
	$\tau(u)=\tau(-\dot{\sigma}(0))$, so by flip invariance we get
	$\tau(u)=\tau(v)$. Therefore $\tau$ is constant.

	The third point is similar to Corollary 2.6 in \cite{Knieper2}. As
	$U'_v(0)-S'_v(0)$ is a positive symmetric matrix, the arithmetic-geometric
	inequality gives
	$$\det(U'_v(0)-S'_v(0))^{\frac{1}{n}}\le\frac{1}{n}\tr(U'_v(0)-S'_v(0))=2h$$
	and the inequality follows.
	The case of equality follows, as in the proof of Corollary 2.6 in
	\cite{Knieper2}, from the fact that  $s\mapsto U'_{\dot{\gamma}(s)}(0)$ and
	$s\mapsto S'_{\dot{\gamma}(s)}(0)$ satisfy the Riccati equation, and because 
	$U'_{\dot{\gamma}(s)}(0) - S'_{\dot{\gamma}(s)}(0) = \frac{2h}{n-1}Id$.
\end{proof}
\begin{rem}
	The second point is very close to Lemma 2.3 and Corollary 2.1 of \cite{He-Kn-Sh},
	but, in our proof, we don't need any lower bound on the curvature and rather use
	Proposition \ref{prop_volum-form}.
\end{rem}

\subsection{Characterization of asymptotic harmonicity}

Proposition \ref{prop_volum-form} says that the volume form of $M$ has purely
exponential growth, with isotropic exponential rate (and is asymptotically perfectly
isotropic when $DR_M$ is bounded). In fact this is a characterization of asymptotic
harmonicity:

\begin{prop} \label{prop_characterization}
	Let $M$ be a CH-manifold with  $K_M \le -a^2<0$ and entropy $E$.  
	If there exists a positive function $\tau:SM\to\R$ such that
	$\theta(u,r)$ is uniformly equivalent to $\tau(u)\ex^{Er}$ for
	$r \to \infty$, then $M$ is asymptotically harmonic.
\end{prop}
\begin{rem}
	Notice that, together with Benoist-Foulon-Labourie and Besson-Courtois-Gallot
	characterization of cocompact asymptotically harmonic spaces,
	Proposition \ref{prop_characterization} shows that {\em if a CH-manifold with
	compact quotients has volume form which is (uniformly) equivalent to a
	function $\tau(u)\ex^{ER}$, then it is a ROSS}.
\end{rem}

\begin{proof}
	Let $\gamma(t)$ be a geodesic of $M$ with
	$\lim_{t\to-\infty}\gamma(t)=\xi\in\dM$,
	and let $h(t)$ be the mean curvature at $\gamma(t)$ of the horosphere $H_\xi(t)$
	centered in $\xi$ and passing through $\gamma(t)$. We shall prove that
	the function $h(t)$ is constant.
	
	Let $r<R$ be two real numbers, and choose $s>-r$. For any $t\in[r,R]$, we use
	Lemma \ref{lem_horosphere} to compare the second fundamental forms of $H_\xi(t)$
	and $S_{\gamma(-s)}(t+s)$ at $\gamma(t)$. Taking the trace in
	(\ref{eqfundamentalforms1}), we have
	$$
	0 \le \frac{\theta'(\dot{\gamma}(-s),t+s)}%
		{\theta(\dot{\gamma}(-s),t+s)} - nh(t) \le na\bigl(\coth(a(t+s))-1\bigr)
	$$
	and integrating on $[r,R]$ with respect to $t$ we get
	\begin{equation} \label{eqn_along-geodesic}
	0 \le \ln\frac{\theta(\dot{\gamma}(-s),R+s)}%
		{\theta(\dot{\gamma}(-s),r+s)} - n\int_r^Rh(t)dt
		\le \ln\Bigl(\frac{\sinh^n(a(R+s))}{\sinh^n(a(r+s))}\ex^{-na(R-r)}\Bigr)
	\end{equation}
	The right-hand side tends to $0$ when $s$ tends to infinity. Moreover, by
	hypothesis we have
	$|\frac{\theta(\dot{\gamma}(-s),R+s)}{\tau(\dot{\gamma}(-s))\ex^{E(R+s)}}-1|
	\le\varepsilon(R+s)$ with $\lim_{s\to\infty}\varepsilon(R+s)=0$, and we get
	$$
	\lim_{s\to\infty}\frac{\theta(\dot{\gamma}(-s),R+s)}%
		{\tau(\dot{\gamma}(-s))\ex^{Es}} = \ex^{ER}.
	$$
	Analogously, we find $\lim_{s\to\infty}\frac{\theta(\dot{\gamma}(-s),r+s)} 
	{\tau(\dot{\gamma}(-s))\ex^{Es}} = \ex^{Er}$, so letting $s$ tend to infinity
	in (\ref{eqn_along-geodesic}) we obtain
	$$
	E(R-r) -n\int_r^Rh(t)dt = 0.
	$$ 	
	Therefore  $\int_r^R(E-nh(t))dt=0$ for all $r<R$, from which we deduce that 
	$h(t)=\frac{E}{n}$  for all $t\in\R$, and $M$ is asymptotically harmonic.
\end{proof}

%
%
%
%
\section{Margulis function and measures at infinity}

In this last section, we assume that $M$ is a asymptotically harmonic CH-manifold
with  pinched curvature $-b^2 \le K_M \le -a^2 < 0$, and $h$ is always  the mean
curvature of the horospheres.

\subsection{Visual and harmonic measures} 
There are two families of measures naturally defined on the ideal boundary
of Cartan-Hadamard manifolds: the visual and harmonic measures.

To define the visual measures, consider the homeomorphism given by the ``projection
on $\dM$ from $x$'':
$$
\phi_x : \left\{\begin{array}{rcl}
					S_xM & \to & \dM \\
					u & \mapsto & \phi_x(u)=\lim_{t\to\infty}\exp_x(tu)
				\end{array}\right.
$$
The measure $\lambda_x$ is the push-forward on $\dM$ of the (normalized)
Riemannian measure of $S_xM$.

On the other hand, the family of  harmonic measures comes from the uniqueness of the
solution to the Dirichlet problem at infinity (cf. \cite{Anderson-Schoen}):
given a continuous function $f$ on $\dM$, there exists a unique bounded
harmonic function $F$ on $M$ such that $\lim_{x\to\xi}F(x)=\xi$. Then, it is a
consequence of Riesz representation theorem that there exists a unique family of
measures
$\mu_x$, $x\in M$, such that $F(x)=\int_{\dM}f(\xi)d\mu_x(\xi)$.

\begin{prop}\label{prop_densities}
	Let $M$ be an asymptotically harmonic CH-manifold with
	pinched curvature $-b^2\le K_M \le -a^2<0$.	
	For any $x,y\in M$ we have
	$$
	\frac{d\lambda_x}{d\lambda_y}(\xi) =
		\frac{\tau(\phi_y^{-1}(\xi))}{\tau(\phi_x^{-1}(\xi))}
		\ex^{-nh(b_\xi(x)-b_\xi(y))}\ \mbox{ and }\ 
		\frac{d\mu_x}{d\mu_y}(\xi) = \ex^{-nh(b_\xi(x)-b_\xi(y))}.
	$$
\end{prop}
\begin{proof}
	Consider the distance functions $r_x$ and $r_y$ to  $x, y \in M$ respectively,
	and the sphere $S_x(t)$ centered in $x$ of radius $t$.
	For $t$ great enough, each geodesic ray from $y$ intersect $S_x(t)$ at a unique
	point; for $v\in S_yM$, let $F_t(v)$ be the intersection point of the geodesic
	$s\mapsto\exp_y(sv)$ and $S_x(t)$.
	
	The map $F_t:S_yM\to S_x(t)$ so defined is a diffeomorphism whose Jacobian is 
	\begin{equation} \label{eqjac}
		\jac_vF_t = \frac{\theta (v,r_y(F_t(v)))}
			{\langle\nabla r_y(F_t(v)),\nabla r_x(F_t(v))\rangle}
	\end{equation}
	
	Now, let $U\subset\dM$ be a measurable set with negligible boundary, and let
	$U_t = \{\exp_x (tu)\ |\ u\in\phi_x^{-1}(U)\} $ be the projection of $U$ on
	$S_x(t)$ from $x$.

	By definition of $\lambda_x$ we have
	$$
	\lambda_x(U) = \int_{\phi_x^{-1} (U)} d\sigma_x
		=\frac{1}{\vol(S^n)}\int_{U_t}\frac{1}{\theta(P_t^{-1}(z),t)}dv_{S_x(t)}(z)
	$$
	where $P_t(u)=\exp_x(tu)$ for $u\in S_xM$, $d\sigma_x$ is the normalized
	measure of $S_xM$, and $dv_{S_x(t)}$ the volume forms of $S_x(t)$.
	
	By (\ref{eqjac}), we get 
	\begin{equation} \label{eqn_lambdax}
		\lambda_x(U) = \int_{F_t^{-1}(U_t)}
			\frac{\theta (v,r_y(F_t(v)))}{\theta(P_t^{-1}\circ F_t(v),t)}
			\langle\nabla r_y(F_t(v)),\nabla r_x(F_t(v))\rangle^{-1}d\sigma_y(v)
	\end{equation}
	where $d\sigma_y$ is the normalized measure on $S_yM$.
	
	Now we observe that, letting $t$ tend  to infinity, we have
	\begin{itemize}
		\item $\lim_{t\to\infty} P_t^{-1}\circ F_t(v) = \phi_x^{-1}\circ\phi_y(v)$;
		\item $\lim_{t\to\infty} \chi_{F_t^{-1}(U_t)} = \chi_{\phi_y^{-1}(U)}$ almost
			everywhere;
		\item $\lim_{t\to\infty} \langle\nabla r_y(F_t(v)),
			\nabla r_x(F_t(v))\rangle=1$
	\end{itemize}
	Moreover, from Theorem \ref{prop_volum-form} we know that
	\begin{multline*}
		\frac{\tau(v)-\varepsilon(r_y(F_t(v)))}
			{\tau(P_t^{-1}\circ F_t(v)))+\varepsilon(t)}\ex^{nh(r_y(F_t(v))-t)}
			\le \frac{\theta_y(v,r_y(F_t(v)))}{\theta(P_t^{-1}\circ F_t(v),t)} \\
		\le \frac{\tau(v)+\varepsilon(r_y(F_t(v)))}
			{\tau(P_t^{-1}\circ F_t(v))-\varepsilon(t)}\ex^{nh(r_y(F_t(v))-t)}
	\end{multline*}
	By definition of  Busemann function  we have that $r_y(F_t(v))-t $ converges,
	uniformly on $S_yM$, to $b_{\phi_y(v)}(y)-b_{\phi_y(v)}(x)$; so, as $\tau$ is
	continuous and bounded, by dominating convergence (\ref{eqn_lambdax}) yields
	\begin{eqnarray*}
		\lambda_x(U) & = & \int_{\phi_y^{-1}(U)}
				\frac{\tau(v)}{\tau(\phi_x^{-1}\circ\phi_y(v))}
				\ex^{-nh(b_{\phi_y(v)}(x)-b_{\phi_y(v)}(y))}dv_{S_yM} (v) \\
		 & = & \int_U
				\frac{\tau(\phi_y^{-1}(\xi))}{\tau(\phi_x^{-1}(\xi))}
				\ex^{-nh(b_\xi(x)-b_\xi(y))}d\lambda_y(\xi)
	\end{eqnarray*}
	which proves the first equality of the proposition.
	
	The second equality follows from \cite{Anderson-Schoen}: the relative
	densities of harmonic measures are given by the Poisson kernel, and, as
	$\Delta b_\xi=-nh$, by unicity of the Poisson kernel we have
	$\frac{d\mu_x}{d\mu_y}(\xi) = \ex^{-nh(b_\xi(x)-b_\xi(y))}$.
\end{proof}

As a consequence of Theorem \ref{prop_tau-properties}, we have that, {\em when the
derivative of the Riemann tensor is bounded, the visual and harmonic measures
class have the same relative densities}.

\subsection{The Margulis function}

For cocompact CH-manifolds Margulis introduced the function 
$$
m(x)=\lim_{r\to\infty}\vol(S_x(r))\ex^{-Er}.
$$
where $E$ is the volume entropy of $M$. The main conjecture concerning this function
is  that it is constant if and only if $M$ is a symmetric space,
cf.\cite{Yue,Knieper} for some related results. 
Theorem \ref{prop_volum-form} allows us to define the Margulis function for
asymptotically harmonic manifolds (even noncocompact) :
\begin{prop} \label{propmargulis}
	Let $M$ be an asymptotically harmonic CH-manifold with  $-b^2\le K \le -a^2<0$. 
	There exists a function $m: M \to \R_+$ such that
	$$
	\lim_{r\to\infty}\vol(S_x(r))\ex^{-nhr} = m(x)
	\ \mbox{ and }\ 
	\lim_{r\to\infty}\vol(B_x(r))\ex^{-nhr} = \frac{m(x)}{nh}
	$$
	for any $x\in M$. Moreover, the function $m$ is harmonic.
\end{prop}
\begin{proof}
	Let $V_x(r)=\Vol(B_x(r))$ and $v_x(r)=\vol(S_x(r))$, so  $V_x^\prime(r)=v_x(r)$.
	Since $v_x(r)=\int_{S_xM}\theta(u,r)du$, integrating (\ref{eqn_tau}) on $S_xM$,
	by monotone convergence we get the first equality with
	$m(x)=\int_{S_xM}\tau(u)du$.
	
	Then, by Proposition \ref{isop},  we have $ V_x^\prime(r)-nhV_x(r) \ge 0$, so 
	$V_x(r)\ex^{-nhr}$ is increasing. As $V_x(r)=\int_0^r\int_{S_xM}\theta(u,s)duds$,
	Theorem \ref{prop_volum-form} implies, for any $r\ge 1$,
	$$
	V_x(r) \le V_x(1) + \frac{m(x)}{nh}(\ex^{nhr}-\ex^{nh}) +
		\vol(S^n)\int_1^r\varepsilon(s)\ex^{nhs}ds
	$$
	from which we deduce that $V_x(r)\ex^{-nhr}$ is bounded; hence, it converges to
	some limit $l(x)$. As  $V_x(r)\ex^{-nhr}$ is increasing and converging, there
	exists a sequence $r_k \to \infty$ such that
	\begin{eqnarray*}
		0 & = & \lim_{k\to\infty}\frac{d}{dr}_{\left| _{r=r_k} \right. }
			\left(V_x(r)\ex^{-nhr} \right) \\
     	& = & \lim_{k\to\infty}(v_x(r_k)\ex^{-nhr_k}-nhV_x(r_k)\ex^{-nhr_k}) \\
     	& = & m(x) - nhl(x).
	\end{eqnarray*}
	
	Finally, to show that the Margulis function is harmonic, we write it using the
	visual measures :
	$$
	m(x) = \int_{S_xM}\tau(u)du = \int_{\dM}\tau(\phi_x^{-1}(\xi))d\lambda_x(\xi) 
	$$
	Choosing a fixed point $x_0\in M$, we get
	$$
	m(x) = \int_{\dM} \tau(\phi_{x_0}^{-1}(\xi))
		\ex^{-nh(b_\xi(x)-b_\xi(x_0))}d\lambda_{x_0}(\xi)
	$$
	and we are done, because $\ex^{-nh(b_\xi(x)}$  is harmonic.
\end{proof}

%
%
%
%

\small
\bibliographystyle{plain}

\bibliography{pc-as-asympt-harm-mfds}

\begin{thebibliography}{10}

\bibitem{Anderson-Schoen}
M.T. Anderson and R.~Schoen.
\newblock Positive harmonic functions on complete manifolds of negative
  curvature.
\newblock {\em Ann. of Math. (2)}, 121(3):429--461, 1985.

\bibitem{Be-Fo-La}
Y.~Benoist, P.~Foulon, and F.~Labourie.
\newblock Flots d'{A}nosov {\`a} distributions stable et instable
  diff{\'e}rentiables.
\newblock {\em J. Amer. Math. Soc.}, 5(1):33--74, 1992.

\bibitem{Besse}
A.L. Besse.
\newblock {\em Manifolds all of whose geodesics are closed}, volume~93 of {\em
  Ergebnisse der Mathematik und ihrer Grenzgebiete [Results in Mathematics and
  Related Areas]}.
\newblock Springer-Verlag, Berlin, 1978.
\newblock With appendices by D. B. A. Epstein, J.-P. Bourguignon, L.
  B{{\'e}}rard-Bergery, M. Berger and J. L. Kazdan.

\bibitem{Be-Co-Ga}
G.~Besson, G.~Courtois, and S.~Gallot.
\newblock Entropies et rigidit{\'e}s des espaces localement sym{\'e}triques de
  courbure strictement n{\'e}gative.
\newblock {\em Geom. Funct. Anal.}, 5(5):731--799, 1995.

\bibitem{Bridson-Haefliger}
M.~R. Bridson and A.~Haefliger.
\newblock {\em Metric spaces of non-positive curvature}, volume 319 of {\em
  Grundlehren der Mathematischen Wissenschaften [Fundamental Principles of
  Mathematical Sciences]}.
\newblock Springer-Verlag, Berlin, 1999.

\bibitem{Connell}
C.~Connell.
\newblock Asymptotic harmonicity of negatively curved homogeneous spaces and
  their measures at infinity.
\newblock {\em Comm. Anal. Geom.}, 8(3):575--633, 2000.

\bibitem{Damek-Ricci}
E.~Damek and F.~Ricci.
\newblock A class of nonsymmetric harmonic {R}iemannian spaces.
\newblock {\em Bull. Amer. Math. Soc. (N.S.)}, 27(1):139--142, 1992.

\bibitem{Donnelly}
H.~Donnelly.
\newblock On the essential spectrum of a complete {R}iemannian manifold.
\newblock {\em Topology}, 20(1):1--14, 1981.

\bibitem{Foulon-Labourie}
P.~Foulon and F.~Labourie.
\newblock Sur les vari{\'e}t{\'e}s compactes asymptotiquement harmoniques.
\newblock {\em Invent. Math.}, 109(1):97--111, 1992.

\bibitem{Heber}
J.~Heber.
\newblock On harmonic and asymptotically harmonic homogeneous spaces.
\newblock {\em Geom. Funct. Anal.}, 16(4):869--890, 2006.

\bibitem{He-Kn-Sh}
J.~Heber, G.~Knieper, and H.~Shah.
\newblock Asymptotically harmonic spaces in dimension 3.
\newblock {\em Proc. Amer. Math. Soc.}, 135(3):845--849 (electronic), 2007.

\bibitem{Heintze-Hof}
E.~Heintze and H.C. Im~Hof.
\newblock Geometry of horospheres.
\newblock {\em J. Differential Geom.}, 12:481--491, 1977.

\bibitem{Karcher}
H.~Karcher.
\newblock Riemannian comparison constructions.
\newblock In {\em Global differential geometry}, volume~27 of {\em MAA Stud.
  Math.}, pages 170--222. Math. Assoc. America, Washington, DC, 1989.

\bibitem{Karp-Peyerimhoff}
L.~Karp and N.~Peyerimhoff.
\newblock Horospherical means and uniform distribution of curves of constant
  geodesic curvature.
\newblock {\em Math. Z.}, 231(4):655--677, 1999.

\bibitem{Knieper}
G.~Knieper.
\newblock Spherical means on compact {R}iemannian manifolds of negative
  curvature.
\newblock {\em Differential Geom. Appl.}, 4(4):361--390, 1994.

\bibitem{Knieper2}
G.~Knieper.
\newblock New results on noncompact harmonic manifolds.
\newblock To appear in Comment. Math. Helv., 2009.

\bibitem{Kumura}
H.~Kumura.
\newblock On the essential spectrum of the {L}aplacian on complete manifolds.
\newblock {\em J. Math. Soc. Japan}, 49(1):1--14, 1997.

\bibitem{Ledrappier}
F.~Ledrappier.
\newblock Harmonic measures and {B}owen-{M}argulis measures.
\newblock {\em Israel J. Math.}, 71(3):275--287, 1990.

\bibitem{Ledrappier2}
F.~Ledrappier.
\newblock Ergodic properties of the stable foliations.
\newblock In {\em Ergodic Theory and Related Topics III}, volume 1514 of {\em
  Lecture Notes in Mathematics}, pages 131--145. Springer, 1992.

\bibitem{Ranjan-Shah}
A.~Ranjan and H.~Shah.
\newblock Harmonic manifolds with minimal horospheres.
\newblock {\em J. Geom. Anal.}, 12(4):683--694, 2002.

\bibitem{Ranjan-Shah2}
A.~Ranjan and H.~Shah.
\newblock Busemann functions in harmonic manifolds.
\newblock {\em Geom. Dedicata}, 101:167--183, 2003.

\bibitem{Reed-Simon}
M.~Reed and B.~Simon.
\newblock {\em Methods of modern mathematical physics. {I}. {F}unctional
  {A}nalysis.}
\newblock Academic Press Inc. [Harcourt Brace Jovanovich Publishers], New York,
  second edition, 1980.
\newblock Functional analysis.

\bibitem{Schroeder-Shah}
V.~Schroeder and H.~Shah.
\newblock On 3-dimensional asymptotically harmonic manifolds.
\newblock {\em Arch. Math. (Basel)}, 90(3):275--278, 2008.

\bibitem{Szabo}
Z.~I. Szab{{\'o}}.
\newblock The {L}ichnerowicz conjecture on harmonic manifolds.
\newblock {\em J. Differential Geom.}, 31(1):1--28, 1990.

\bibitem{Yau}
S.T. Yau.
\newblock Harmonic functions on complete {R}iemannian manifolds.
\newblock {\em Comm. Pure Appl. Math.}, 28:201--228, 1975.

\bibitem{Yue2}
C.~Yue.
\newblock Rigidity and dynamics around manifolds of negative curvature.
\newblock {\em Math. Res. Lett.}, 2(1):123--147, 1994.

\bibitem{Yue}
C.~Yue.
\newblock Brownian motion on {A}nosov foliations and manifolds of negative
  curvature.
\newblock {\em J. Differential Geom.}, 41(1):159--183, 1995.

\end{thebibliography}

\vspace{10mm}

  \begin{flushleft}
    Philippe Castillon \\
    \textsc{i3m} (\textsc{u.m.r. c.n.r.s.} 5149) \\
    D\'ept. des Sciences Math\'ematiques, CC 51 \\
    Univ. Montpellier II \\
    34095 \textsc{Montpellier} Cedex 5, France \\
    \texttt{philippe.castillon\symbol{64}univ-montp2.fr}
  \end{flushleft}

  \begin{flushleft}
    Andrea Sambusetti \\
    Istituto di Matematica G. Castelnuovo \\
    Universit\`a "La Sapienza" di Roma \\
    P.le Aldo Moro 5 \\
    00185 Roma (Italy) \\
    \texttt{sambuset\symbol{64}mat.uniroma1.it}
  \end{flushleft}

\end{document}